\documentclass[10pt]{amsart}
\usepackage{amssymb,MnSymbol}
\usepackage{amsthm,amsmath}

\title[]{A classification of barycentrically associative polynomial functions}\thanks{Corresponding author: Jean-Luc Marichal is with the Mathematics Research Unit, University of Luxembourg, 6, rue Coudenhove-Kalergi, L-1359 Luxembourg, Luxembourg.\\ Email: jean-luc.marichal[at]uni.lu}

\author{Jean-Luc Marichal}
\address{Mathematics Research Unit, FSTC, University of Luxembourg, 6, rue Coudenhove-Kalergi, L-1359 Luxembourg, Luxembourg}
\email{jean-luc.marichal[at]uni.lu }

\author{Pierre Mathonet}
\address{University of Li\`ege, Department of Mathematics, Grande Traverse, 12 - B37, B-4000 Li\`ege, Belgium}
\email{p.mathonet[at]ulg.ac.be }

\author{J\"org Tomaschek}
\address{Mathematics Research Unit, FSTC, University of Luxembourg, 6, rue Coudenhove-Kalergi, L-1359 Luxembourg, Luxembourg}
\email{research[at]jtomaschek.eu }

\date{October 31, 2014}

\begin{document}

\theoremstyle{plain}
\newtheorem{theorem}{Theorem}%[section]% Supprimer [section] pour une numérotation linéaire
\newtheorem{lemma}[theorem]{Lemma}
\newtheorem{proposition}[theorem]{Proposition}
\newtheorem{corollary}[theorem]{Corollary}
\newtheorem{fact}[theorem]{Fact}
\newtheorem*{main}{Main Theorem}

\theoremstyle{definition}
\newtheorem{definition}[theorem]{Definition}
\newtheorem{example}[theorem]{Example}

\theoremstyle{remark}
\newtheorem*{conjecture}{Conjecture}
\newtheorem*{remark}{Remark}
\newtheorem*{claim}{Claim}

\newcommand{\N}{\mathbb{N}}
\newcommand{\Z}{\mathbb{Z}}
\newcommand{\R}{\mathbb{R}}
\newcommand{\C}{\mathbb{C}}
\newcommand{\RR}{\mathcal{R}}
\newcommand{\bfx}{\mathbf{x}}
\newcommand{\bfy}{\mathbf{y}}
\newcommand{\bfz}{\mathbf{z}}
\newcommand{\bfalpha}{\boldsymbol{\alpha}}
\newcommand{\Ast}{\boldsymbol{\ast}}
\newcommand{\Cdot}{\boldsymbol{\cdot}}

\begin{abstract}
We describe the class of polynomial functions which are barycentrically associative over an infinite commutative integral domain.
\end{abstract}

\keywords{Barycentric associativity, decomposability, polynomial function, integral domain.}

\subjclass[2010]{Primary 39B72; Secondary 13B25, 26B35.}

\maketitle

%---------------------------------------------------------------------------------------------- Section 1
\section{Introduction}

Let $\N$ be the set of nonnegative integers. Let also $X$ be an arbitrary nonempty set and let $X^*=\bigcup_{n\in\N} X^n$ be the set of all tuples on $X$, with the convention that $X^0=\{\varepsilon\}$ (i.e., $\varepsilon$ denotes the unique $0$-tuple on $X$). As usual, a function $F\colon X^n\to X$ is said to be \emph{$n$-ary}. Similarly, we say that a function $F\colon X^*\to X$ is \emph{$\Ast$-ary}. With a slight abuse of notation we may assume that every $\Ast$-ary function $F\colon X^*\to X$ satisfies $F(\varepsilon)=\varepsilon$. The \emph{$n$-ary part} $F_n$ of a function $F\colon X^*\to X$ is the restriction of $F$ to $X^n$, that is, $F_n=F|_{X^n}$. For tuples $\bfx=(x_1,\ldots,x_n)$ and $\bfy=(y_1,\ldots,y_m)$, the notation $F(\bfx,\bfy)$ stands for $F(x_1,\ldots,x_n,y_1,\ldots,y_m)$, and similarly for more than two tuples.

A function $F\colon X^*\to X$ is said to be \emph{barycentrically associative}, or \emph{B-associative} for short, if
\begin{equation}\label{eq:B-asso}
F(\bfx,\bfy,\bfz) ~=~ F(\bfx,k\Cdot F(\bfy),\bfz),
\end{equation}
for every integer $k\in\N$ and every $\bfx,\bfz\in X^*$ and $\bfy\in X^k$, where the notation $k\Cdot x$ means that the argument $x$ is repeated $k$ times. For instance, $F(x,2\Cdot y)=F(x,y,y)$.

Barycentric associativity was introduced in Schimmack \cite{Sch09} as a natural and suitable variant of associativity to characterize the arithmetic mean. Contrary to associativity, this property is satisfied by various means, including the geometric mean and the harmonic mean. It was also used by Kolmogoroff \cite{Kol30} and Nagumo \cite{Nag30} to characterize the class of quasi-arithmetic means.

Since its introduction this property was used under at least three different names: \emph{associativity of means} \cite{deF31}, \emph{decomposability} \cite[Sect.~5.3]{FodRou94}, and \emph{barycentric associativity} \cite{Ant98}. Here we have chosen the third one, which naturally recalls the associativity property of the barycenter as defined in affine geometry. For general background on barycentric associativity and its links with associativity, see \cite[Sect.~2.3]{GraMarMesPap09}.

Let $\RR$ be an infinite commutative integral domain (with identity). We say that a function $F\colon\RR^*\to\RR$ is a \emph{$\ast$-ary polynomial function}, or simply a \emph{polynomial function}, if $F_n=F|_{\RR^n}$ is a polynomial function for every integer $n\geqslant 1$.

In this note we provide a complete description of those polynomial functions $F\colon\RR^*\to\RR$ which are B-associative. This description is given in the Main Theorem below and the proof is given in the next section.

Any polynomial function $F\colon\RR^*\to\RR$ such that $F_n$ is constant for every $n\geqslant 1$ is clearly B-associative. It is straightforward to see that nontrivial instances of B-associative polynomial functions include
\begin{itemize}
\item the first projection, defined by $F_n(x_1,\ldots,x_n)=x_1$ for every $n\geqslant 1$,

\item the last projection, defined by $F_n(x_1,\ldots,x_n)=x_n$ for every $n\geqslant 1$,

\item the arithmetic mean, defined by $F_n(x_1,\ldots,x_n)=n^{-1}\sum_{i=1}^nx_i$ for every $n\geqslant 1$ (assuming that every integer $n\geqslant 1$ is invertible in $\RR$).
\end{itemize}

These examples are special cases of the following one-parameter family of polynomial functions. For every integer $n\geqslant 1$ and every $z\in\RR$ such that
$$
\Delta_n^z ~=~\sum_{i=1}^nz^{n-i}(1-z)^{i-1} ~=~\Delta_n^{1-z}
$$
is invertible, define the weighted arithmetic mean function $M_n^z\colon\RR^n\to\RR$ by
$$
M_n^z(\bfx) ~=~ (\Delta_n^z)^{-1}\,\sum_{i=1}^nz^{n-i}(1-z)^{i-1}{\,}x_i{\,}.
$$
For every $z\in\RR$ we define
$$
n(z) ~=~ \inf\{n\geqslant 1 : \Delta_n^z~\mbox{is not invertible}\}.
$$
Clearly, we have $n(z)\geqslant 3$. If $\Delta_n^z$ is invertible for every integer $n\geqslant 1$, then we set $n(z)=\infty$.

For every $z\in\RR$, consider the function $M^z\colon\RR^*\to\RR$ whose restriction to $\RR^n$ is $M_n^z$ if $n<n(z)$, and $0$, otherwise. The Main Theorem states that, up to special cases and constant functions, the typical B-associative polynomial functions are the functions $M^z$, where $z\in\RR$. Note that the special functions $M^1$, $M^0$, and $M^{1/2}$ are precisely the three above-mentioned instances of B-associative polynomial functions.

Given a function $F\colon X^*\to X$ and an integer $k\geqslant 1$ or $k=\infty$, we denote by $[F]_k$ the class of functions $G\colon X^*\to X$ obtained from $F$ by replacing $F_n$ with a constant function for every $n\geqslant k$. In particular, we have $[F]_{\infty}=\{F\}$.

\begin{main}
A polynomial function $F\colon\RR^*\to\RR$ is B-associative if and only if one of the following two conditions holds.
\begin{enumerate}
\item[(i)] There exist $z\in\RR$ and an integer $k\geqslant 1$ or $k=\infty$, with $k\leqslant n(z)$, such that $F\in [M^z]_k$.

\item[(ii)] There exists a polynomial function $Q\colon\RR^2\to\RR$ of degree $\geqslant 1$ such that $F_1(x)=x$, $F_2(x,y)=Q(x,y){\,}x+(1-Q(x,y)){\,}y$, and $F_n$ is constant for every $n\geqslant 3$.
\end{enumerate}
\end{main}

\begin{remark}
By the very definition of function $M^z$, we see that the condition $k\leqslant n(z)$ is not really needed to describe the set of possible functions $F$ in case (i) of the Main Theorem. However, we have added this condition to stress on the fact that $F_n$ can be any constant function for every $n\geqslant n(z)$.
\end{remark}

\begin{example}
Suppose that $\RR$ is a field of characteristic zero. One can readily see that $\Delta_n^z=0$ if and only if $(1-z)^n=z^n$ and $2z-1\neq 0$, that is, if and only if $z=1/(1+\omega_n)$, where $\omega_n\in\RR\setminus\{-1,1\}$ is an $n$-th root of unity. For instance, if $\RR$ is the field $\C$ of complex numbers and $F\colon\C^*\to\C$ is a B-associative polynomial function such that $F_3=M_3^z$, with $z=1/(1+i)$, then necessarily $F_n$ is constant for every $n\geqslant 4$.
\end{example}

\begin{example}
If $\RR$ is the ring $\Z$ of integers, then $n(0)=n(1)=\infty$ and $n(z)=3$ for every $z\in\Z\setminus\{0,1\}$. Thus, if $F\colon\Z^*\to\Z$ is a B-associative polynomial function of type (i), then $F\in [M^0]_k$ or $F\in [M^1]_k$ for some integer $k\geqslant 1$ or $k=\infty$, or $F\in [M^z]_k$ for some $z\in\Z\setminus\{0,1\}$ and some $k\in\{1,2,3\}$.
\end{example}

The following straightforward corollary concerns the special case when $F_n$ is symmetric (i.e., invariant under any permutation of the arguments) for every $n\geqslant 1$.

\begin{corollary}\label{cor:Sym}
Let $F\colon\RR^*\to\RR$ be a polynomial function such that $F_n$ is symmetric for every $n\geqslant 1$. Then $F$ is B-associative if and only if either $F_n$ is constant for every $n\geqslant 1$ or $1/2\in\RR$ and one of the following two conditions holds.
\begin{enumerate}
\item[(i)] There exists an integer $k\geqslant 2$ or $k=\infty$, with $k\leqslant n(1/2)$, such that $F\in [M^{1/2}]_k$.

\item[(ii)] There exists a nonzero antisymmetric polynomial function $Q\colon\RR^2\to\RR$ such that $F_1(x)=x$, $F_2(x,y)=\frac{x+y}{2}+(x-y){\,}Q(x,y)$, and $F_n$ is constant for every $n\geqslant 3$.
\end{enumerate}
\end{corollary}

%---------------------------------------------------------------------------------------------- Section 2
\section{Technicalities and proof of the Main Theorem}

We observe that the definition of $\RR$ enables us to identify the ring $\RR[x_1,\ldots,x_n]$ of polynomials of $n$ indeterminates over $\RR$ with the ring of polynomial functions of $n$ variables from $\RR^n$ to $\RR$.

It is a straightforward exercise to show that the $\ast$-ary polynomial functions given in the Main Theorem are B-associative.

We now show that no other $\ast$-ary polynomial function is B-associative. We first consider the special case when $\RR$ is a field. We will then prove the Main Theorem in the general case (i.e., when $\RR$ is an infinite commutative integral domain).

From the definition of B-associative functions, we immediately derive the following interesting fact.

\begin{fact}\label{fact:1}
Let $F\colon X^*\to X$ be a B-associative function.
\begin{enumerate}
\item[(i)] If $F_n$ is constant for some $n\geqslant 1$, then so is $F_{n+1}$.

\item[(ii)] Any $G\in \bigcup_{k\geqslant 1}[F]_k$ is B-associative.
\end{enumerate}
\end{fact}

A function $F\colon X^n\to X$ is said to be \emph{idempotent} if $F(n\Cdot x)=x$ for every $x\in X$. It is said to be \emph{range-idempotent} if $F(n\Cdot x)=x$ for every $x$ in the range of $F$. Equivalently, $F$ is range-idempotent if $\delta_F\circ F=F$, where $\delta_F$ is the diagonal section of $F$, defined by $\delta_F(x)=F(n\Cdot x)$. In this case we clearly have $\delta_F\circ\delta_F=\delta_F$.

Now let $F\colon\RR^*\to\RR$ be a B-associative polynomial function, where $\RR$ is a field. Since $F$ is B-associative, $F_n$ is clearly range-idempotent for every $n\geqslant 1$ (just take $\bfx =\bfz =\varepsilon$ in Eq.~(\ref{eq:B-asso})). The following lemma then shows that $F_n$ is either constant or idempotent.

\begin{lemma}\label{lemma:123}
A polynomial function $F\colon\RR^n\to\RR$ is range-idempotent if and only if it is either constant or idempotent.
\end{lemma}

\begin{proof}
The condition is trivially sufficient. To see that it is also necessary, we let $F\colon\RR^n\to\RR$ be a range-idempotent polynomial function and show that its diagonal section $\delta_F$ is either constant or the identity function. Clearly, if $\delta_F$ is constant, then so is $F=\delta_F\circ F$.

Suppose that $\delta_F$ is nonconstant and let us write $\delta_F(x) = \sum_{i=0}^d a_i x^i$, with $d\geqslant 1$ and $a_d\neq 0$. By equating the leading (i.e., highest degree) terms in both sides of the identity $\delta_F\circ\delta_F=\delta_F$, we obtain $a_d^2x^{d^2}=a_dx^d$. Therefore, we must have $d=1$ and $a_1=1$, that is, $\delta_F(x)=x+a_0$. Substituting again in $\delta_F\circ\delta_F=\delta_F$, we obtain $a_0=0$.
\end{proof}

Let us write $F_n$ is the following standard form
$$%\begin{equation}\label{eq:pol}
F_n(\bfx) ~=~ \sum_{j=0}^d\sum_{|\bfalpha|=j}a_{\bfalpha}\,\bfx^{\bfalpha},\quad \mbox{with}~\bfx^{\bfalpha} = x_1^{\alpha_1}\cdots\, x_n^{\alpha_n},
$$%\end{equation}
where the inner sum is taken over all $\bfalpha\in\N^n$ such that $|\bfalpha|=\alpha_1+\cdots +\alpha_n=j$. This polynomial function is said to be of \emph{degree} $d$ if there exists $\bfalpha\in\N^n$, with $|\bfalpha|=d$, such that $a_{\bfalpha}\neq 0$.

Due to Fact~\ref{fact:1}, we may always assume that $F_n$ is nonconstant. By Lemma~\ref{lemma:123}, it is therefore idempotent, which means that
$$
\sum_{j=0}^d\left(\sum_{|\bfalpha|=j}a_{\bfalpha}\right)\, x^{j} ~=~ x,\qquad x\in\RR,
$$
or equivalently,
$$
\sum_{|\bfalpha|=1}a_{\bfalpha} ~=~ 1\quad\mbox{and}\quad \sum_{|\bfalpha|=j}a_{\bfalpha} ~=~ 0\quad \mbox{for $j\neq 1$}.
$$

We then have the following results.

\begin{lemma}\label{lemma:123xxe}
Let $F\colon\RR^*\to\RR$ be a B-associative polynomial function and assume that $F_{n+1}$ is nonconstant for some $n\geqslant 2$. Then there exists an idempotent binary polynomial function $P\colon\RR^2\to\RR$ such that
\begin{eqnarray}
F_{n+1}(x_1,\ldots,x_{n+1}) &=& P(F_n(x_1,\ldots,x_n),x_{n+1}),\label{eq:anc.2.5nF}\\
%P(F_n(x_1,\ldots,x_n),x_{n+1})
&=& P(F_n(x_1,(n-1)\Cdot F_n(x_2,\ldots,x_{n+1})),F_n(x_2,\ldots,x_{n+1}))\label{eq:anc.2.5n}
\end{eqnarray}
and
\begin{equation}\label{eq:anc.11n}
P(F_n(F_n(x_2,\ldots,x_{n+1}),x_2,\ldots,x_n),x_{n+1}) ~=~ F_n(x_2,\ldots,x_{n+1}).
\end{equation}
\end{lemma}

\begin{proof}
Consider the binary polynomial functions $P\colon\RR^2\to\RR$ and $Q\colon\RR^2\to\RR$ defined by $P(x,y)=F_{n+1}(n\Cdot x,y)$ and $Q(x,y)=F_{n+1}(x,n\Cdot y)$, respectively. Since $F_{n+1}$ is nonconstant, by Lemma~\ref{lemma:123} it must be idempotent and therefore so are $P$ and $Q$. By B-associativity of $F$, we then obtain Eq.~(\ref{eq:anc.2.5nF}) and
\begin{equation}\label{eq:PQ-BAn}
P(F_n(x_1,\ldots,x_n),x_{n+1}) ~=~ Q(x_1,F_n(x_2,\ldots,x_{n+1})).
\end{equation}
Clearly, $F_n$ is nonconstant by Fact~\ref{fact:1}. Setting $x_{n+1}=x_n=\cdots =x_2$ in Eq.~(\ref{eq:PQ-BAn}) and then using idempotence, we obtain
$$
P(F_n(x_1,(n-1)\Cdot x_2),x_2) ~=~ Q(x_1,x_2).
$$
Then, substituting for $Q$ in Eq.~(\ref{eq:PQ-BAn}) from the latter equation, we obtain Eq.~(\ref{eq:anc.2.5n}). Finally, setting $x_1=F_n(x_2,\ldots,x_{n+1})$ in either Eq.~(\ref{eq:anc.2.5n}) or Eq.~(\ref{eq:PQ-BAn}) and then using idempotence, we obtain Eq.~(\ref{eq:anc.11n}).
\end{proof}

\begin{proposition}\label{prop:3}
Let $F\colon\RR^*\to\RR$ be a B-associative polynomial function. If $F_3$ is nonconstant, then $F_2$ must be of degree $1$.
\end{proposition}

\begin{proof}
%To establish Proposition~\ref{prop:3}, we first consider the following claim.
%
%\begin{claim}\label{claim:1}
%There exist idempotent binary polynomial functions $P\colon\RR^2\to\RR$ and $Q\colon\RR^2\to\RR$ such that
%\begin{equation}\label{eq:PQ-BA}
%P(F_2(x_1,x_2),x_3) ~=~ Q(x_1,F_2(x_2,x_3)).
%\end{equation}
%Moreover, $P$ satisfies the equations
%\begin{equation}\label{eq:anc.2.5}
%P(F_2(x_1,x_2),x_3) ~=~ P(F_2(x_1,F_2(x_2,x_3)),F_2(x_2,x_3)).
%\end{equation}
%\end{claim}
%
%\begin{proof}[Proof of Claim~\ref{claim:1}]
%Define $P(x,y)=F_3(2\Cdot x,y)$ and $Q(x,y)=F_3(x,2\Cdot y)$. Since $F_3$ is nonconstant, it must be idempotent and hence so are $P$ and $Q$. Eq.~(\ref{eq:PQ-BA}) then follows from the B-associativity of $F$. Now, setting $x_3=x_2$ in Eq.~(\ref{eq:PQ-BA}) and using idempotence, we obtain $P(F_2(x_1,x_2),x_2) = Q(x_1,x_2)$. Then, substituting for $Q$ in Eq.~(\ref{eq:PQ-BA}) from the latter equation, we finally obtain Eq.~(\ref{eq:anc.2.5}).
%\end{proof}
Let us particularize Lemma~\ref{lemma:123xxe} to the case $n=2$. There exists an idempotent binary polynomial function $P\colon\RR^2\to\RR$ such that
\begin{equation}\label{eq:anc.2.5}
P(F_2(x_1,x_2),x_3) ~=~ P(F_2(x_1,F_2(x_2,x_3)),F_2(x_2,x_3))
\end{equation}
and
\begin{equation}\label{eq:anc.11}
P(F_2(F_2(x_2,x_3),x_2),x_3) - F_2(x_2,x_3) ~=~ 0.
\end{equation}
Clearly, $F_2$ is nonconstant by Fact~\ref{fact:1}. Let us express $F_2$ and $P$ in the following convenient ways. Let $p$ (resp.\ $q$) be the degree of $P$ (resp.\ $F_2$) in the first variable. Then there are polynomial functions $P_i\colon\RR\to\RR$ $(i=0,\ldots,p)$ and $Q_j\colon\RR\to\RR$ $(j=0,\ldots,q)$, with $P_p\neq 0$ and $Q_q\neq 0$, such that
\begin{equation}\label{eq:PpQq}
P(x,y) ~=~ \sum_{i=0}^p x^i{\,} P_i(y)\quad\mbox{and}\quad F_2(x,y) ~=~ \sum_{j=0}^q x^j{\,} Q_j(y).
\end{equation}
Considering the standard form of $F_2$, we can also write
$$
F_2(x,y) ~=~ \sum_{k+\ell\leqslant d}a_{k,\ell}{\,}x^ky^{\ell} ~=~ \sum_{m=0}^d R_m(x,y),
$$
where $d$ is the degree of $F_2$ and
$$
R_m(x,y) ~=~ \sum_{k+\ell = m}a_{k,\ell}{\,}x^ky^{\ell},\quad\mbox{with}~ R_d\neq 0.
$$

\begin{claim}\label{claim:2}
If $p>0$ and $q>0$, then the polynomial functions $P_p$ and $Q_q$ are constant.
\end{claim}

\begin{proof}%[Proof of Claim~\ref{claim:2}]
Substituting for $P$ and $F_2$ from Eq.~(\ref{eq:PpQq}) in Eq.~(\ref{eq:anc.2.5}) and then equating the leading terms in $x_1$ in the resulting equation, we obtain
$$
(x_1^q{\,}Q_q(x_2))^p{\,}P_p(x_3) ~=~ (x_1^q{\,}Q_q(F_2(x_2,x_3)))^p{\,}P_p(F_2(x_2,x_3)),
$$
or, equivalently, $G(x_2,x_3)-H(x_2,x_3)=0$, where
$$
G(x_2,x_3) ~=~ Q_q^p(F_2(x_2,x_3)){\,}P_p(F_2(x_2,x_3))\quad\mbox{and}\quad H(x_2,x_3)~=~ Q_q^p(x_2){\,}P_p(x_3).
$$
Denote by $ax^{\alpha}$ (resp.\ $bx^{\beta}$) the leading term of $P_p$ (resp.\ $Q_q$); hence $ab\neq 0$. Clearly, the leading term in $x_2$ of $G$ is
\begin{equation}\label{eq:leadGx2}
(b(x_2^q{\,}Q_q(x_3))^{\beta})^p{\,}a(x_2^q{\,}Q_q(x_3))^{\alpha}
\end{equation}
and is therefore of degree $pq\beta +q\alpha$. Similarly, the leading term in $x_2$ of $H$ is
$$
(bx_2^{\beta})^p{\,}P_p(x_3)
$$
and is of degree $p\beta$.

If $pq\beta +q\alpha > p\beta$, then the expression in Eq.~(\ref{eq:leadGx2}) must be the zero polynomial function, which is impossible since $Q_q\neq 0$. Therefore we must have $pq\beta +q\alpha = p\beta$, that is $\alpha=0$ (i.e., $P_p$ is the constant $a$) and $(q-1)\beta =0$. If $q=1$, then the leading term in $x_2$ of $G(x_2,x_3)-H(x_2,x_3)$ is
$$
(b(x_2{\,}Q_q(x_3))^{\beta})^p{\,}a-(bx_2^{\beta})^p{\,}a ~=~ (bx_2^{\beta})^p{\,}a{\,}(Q_q(x_3)^{p\beta}-1){\,},
$$
and hence $Q_q$ must be constant.
%For contradiction, suppose that $P_p$ or $Q_q$ is nonconstant, that is, $(\alpha,\beta)\neq (0,0)$. It follows that $pq\beta +q\alpha > p\beta$. The expression in Eq.~(\ref{eq:leadGx2}) must therefore be the zero polynomial function, which is impossible since $Q_q\neq 0$. This completes the proof of the claim.
\end{proof}

Let us now prove that $F_2$ is of degree $1$. We consider the following cases, which cover all the possibilities.
\begin{description}
\item[Case $q=0$] We have $F_2(x,y)=Q_0(y)$ and therefore $y=F_2(y,y)=Q_0(y)=F_2(x,y)$, which shows that $F_2$ is of degree $1$.

\item[Case $p=0$] We have $P(x,y)=P_0(y)$. Using idempotence, we obtain $y=P(y,y)=P_0(y)$ and therefore $P(x,y)=y$. Substituting for $P$ in Eq.~(\ref{eq:anc.2.5}), we obtain $x_3=F_2(x_2,x_3)$ and therefore $F_2$ is of degree $1$.

\item[Case $p>0$ and $q=1$] We have $F_2(x_1,x_2)=x_1Q_1(x_2)+Q_0(x_2)$ with $Q_1\neq 0$. Since $F_2$ is idempotent, we also have $x=F_2(x,x)=xQ_1(x)+Q_0(x)$. But $Q_1$ is constant by the claim. It follows that $Q_0$ is of degree $1$ and therefore so is $F_2$.

\item[Case $p>0$ and $q>1$] By definition of $q$ we must have $d\geqslant 2$. Let us compute the leading terms (i.e., homogeneous terms of highest degree) of the left-hand side of Eq.~(\ref{eq:anc.11}). On the one hand, we have
$$
F_2(F_2(x_2,x_3),x_2) ~=~ \sum_{k+\ell\leqslant d} a_{k,\ell}\underbrace{\left(\sum_{m=0}^d R_m(x_2,x_3)\right)^kx_2^{\ell}}_{(*)}{\,},
$$
where the expression $(*)$ is of degree $kd+\ell$, with $R_d^k(x_2,x_3){\,}x_2^{\ell}$ as leading terms. We also have
$$
\max\{kd+\ell : k+\ell\leqslant d,~~ a_{k,\ell}\neq 0\} ~=~ qd.
$$
Indeed, if $k>q$, then $a_{k,\ell}=0$ by definition of $q$. If $k=q$ and $\ell\neq 0$, then $a_{k,\ell}=0$ by the claim. If $k=q$ and $\ell=0$, then $a_{k,\ell}\neq 0$ and $kd+\ell = qd$. Finally, if $k\leqslant q-1$, then
\begin{eqnarray*}
kd+\ell &\leqslant & kd+d-k ~=~ k(d-1)+d ~\leqslant ~ (q-1)(d-1)+d\\
&=& qd-q+1 ~< ~ qd\quad (\mbox{since $q>1$}).
\end{eqnarray*}
This shows that the leading terms of $F_2(F_2(x_2,x_3),x_2)$ are of degree $qd$ and consist of $a_{q,0}{\,}R_d^q(x_2,x_3)$, where $a_{q,0}\neq 0$.

Now, to compute the leading terms of $P(F_2(F_2(x_2,x_3),x_2),x_3)$, it is convenient to express $P$ as
$$
P(x,y) ~=~ \sum_{rqd+s\leqslant e}b_{r,s}{\,}x^ry^{s} ~=~ \sum_{m=0}^e S_m(x,y),
$$
where $e=\max\{rqd+s : b_{r,s}\neq 0\}$ and
$$
S_m(x,y) ~=~ \sum_{rqd+s=m}b_{r,s}{\,}x^ry^{s},\quad\mbox{with}~ S_e\neq 0.
$$
It follows that the leading terms of $P(F_2(F_2(x_2,x_3),x_2),x_3)$ are of degree $e$ and consist of $S_e(a_{q,0}{\,}R_d^q(x_2,x_3),x_3)$. On the other hand, the leading terms of $F_2(x_2,x_3)$ are of degree $d$ and consist of $R_d(x_2,x_3)$.

We observe that there exists $r>0$ such that $b_{r,s}\neq 0$ (otherwise, if $b_{r,s}=0$ for every $r>0$, then $p=0$, a contradiction). By definition of $e$, we then have $e\geqslant rdq>d$. By Eq.~(\ref{eq:anc.11}), we then have $S_e(a_{q,0}{\,}R_d^q(x_2,x_3),x_3) = 0$, or equivalently,
\begin{equation}\label{eq:anc.11a}
\sum_{rqd+s=e}b_{r,s}{\,}\left(a_{q,0}{\,}R_d^q(x_2,x_3)\right)^rx_3^s ~=~ 0.
\end{equation}
Since $R_d(x_2,x_3)$ is of degree $\geqslant 1$ in $x_2$ (otherwise, we would have $R_d(x,y)=T(y)$ and therefore $0=R_d(y,y)=T(y)=R_d(x,y)$, a contradiction), we can write
$$
R_d(x_2,x_3) ~=~ \sum_{k=0}^{f}x_2^k{\,}T_k(x_3),\quad\mbox{with $f>0$ and $T_f\neq 0$}.
$$
Equating the leading terms in $x_2$ in Eq.~(\ref{eq:anc.11a}), we obtain
$$
b_{r_0,e-r_0qd}{\,}\left(a_{q,0}{\,}x_2^{fq}{\,}T_f^q(x_3)\right)^{r_0}x_3^{e-r_0qd} ~=~ 0,
$$
where $r_0 = \max\{r : rqd+s=e,~b_{r,s}\neq 0\}$. This is a contradiction.
\end{description}
This completes the proof of the proposition.
\end{proof}

\begin{proposition}\label{prop:new}
Let $F\colon\RR^*\to\RR$ be a B-associative polynomial function. If $F_n=M_n^z$ for some $n\geqslant 2$ and some $z\in\RR$ such that $\Delta_n^z\neq 0$, then either $F_{n+1}=M_{n+1}^z$ or $F_{n+1}$ is constant. Moreover, if $\Delta_{n+1}^z=0$, then $F_{n+1}$ is constant.
\end{proposition}

\begin{proof}
Assume that $F_n=M_n^z$ for some $n\geqslant 2$ and some $z\in\RR$ such that $\Delta_n^z\neq 0$ and assume that $F_{n+1}$ is nonconstant. Substituting in Eq.~(\ref{eq:anc.11n}) and observing that $(1-z){\,}\Delta_n^z+z^n=\Delta_{n+1}^z$, we obtain
\begin{equation}\label{eq:aia17}
P\left(\Delta_{n+1}^z{\,}\sum_{i=2}^n\frac{z^{n-i}{\,}(1-z)^{i-2}}{(\Delta_n^z)^2}{\,}x_i
+\frac{z^{n-1}(1-z)^{n-1}}{(\Delta_n^z)^2}{\,}x_{n+1},x_{n+1}\right) ~=~ \sum_{i=1}^n\frac{z^{n-i}{\,}(1-z)^{i-1}}{\Delta_n^z}{\,}x_{i+1}.
\end{equation}
If $z=0$, then Eq.~(\ref{eq:aia17}) reduces to $P(x_n,x_{n+1})=x_{n+1}$. By Eq.~(\ref{eq:anc.2.5nF}), we obtain $F_{n+1}(x_1,\ldots,x_{n+1})=x_{n+1}$, that is, $F_{n+1}=M_{n+1}^z$. We can henceforth assume that $z\neq 0$.

If $\Delta_{n+1}^z=0$, then we obtain a contradiction; indeed, the left-hand side of Eq.~(\ref{eq:aia17}) is independent of $x_2$ whereas the coefficient of $x_2$ in the right-hand side is $z^{n-1}/\Delta_n^z$. In this case $F_{n+1}$ must be constant.

We can now assume that $\Delta_{n+1}^z\neq 0$. Using the expression of $P$ given in Eq.~(\ref{eq:PpQq}) and equating the leading terms in $x_2$ in Eq.~(\ref{eq:aia17}), we obtain
$$
\left(\frac{\Delta_{n+1}^z}{(\Delta_n^z)^2}{\,} z^{n-2}{\,}x_2\right)^p P_p(x_{n+1}) ~=~ \frac{z^{n-1}}{\Delta_n^z}{\,}x_2.
$$
It follows that $p=1$ and that $P_1$ is constant, say $P_1=c$, where $c=z{\,}\Delta_n^z/\Delta_{n+1}^z$. We then have $P(x,y)=cx+P_0(y)$ and, by idempotence of $P$, we also have $cx+P_0(x)=x$. Therefore, $P(x,y)=cx+(1-c)y$. Finally, by Eq.~(\ref{eq:anc.2.5nF}) we obtain
$$
F_{n+1}(x_1,\ldots,x_{n+1}) ~=~ c{\,}F_n(x_1,\ldots,x_n) + (1-c){\,} x_{n+1} ~=~ M_{n+1}^z{\,}.
$$
This completes the proof of the proposition.
\end{proof}

Let us now show that any B-associative polynomial function $F\colon\RR^*\to\RR$, where $\RR$ is a field, falls into one of the two cases given in the Main Theorem.

Suppose first that $F_1$ or $F_2$ is constant. In the latter case, $F_1$ is either constant or the identity function by Lemma~\ref{lemma:123}. By Fact~\ref{fact:1}, $F_n$ is constant for every $n\geqslant 2$ and therefore $F$ falls into case $(i)$ with $k=1$ or $k=2$.

Suppose now that $F_1$ and $F_2$ are nonconstant. These functions are idempotent by Lemma~\ref{lemma:123} and therefore $F_1$ is the identity function. If $F_2$ is of degree $1$, then by Lemma~\ref{lemma:123} we have $F_2(x,y)=zx+(1-z)y$ for some $z\in\RR$ and therefore $F$ falls into case (i) by Propositions~\ref{prop:new} and Fact~\ref{fact:1}. Otherwise if $F_2$ is of degree $\geqslant 2$, then by Proposition~\ref{prop:3} and Fact~\ref{fact:1} we have $F_1(x)=x$, $F_2(x,y)=zx+(1-z)y+R(x,y)$ for some $z\in\RR$ and some polynomial function $R\colon\RR^2\to\RR$ of degree $\geqslant 2$ such that $R(x,x)=0$ for all $x\in\RR$, and $F_n$ is constant for every $n\geqslant 3$. It is easy to see that a polynomial function $R\colon\RR^2\to\RR$ satisfies $R(x,x)=0$ for all $x\in\RR$ if and only if we have $R(x,y)=(x-y){\,}Q'(x,y)$ for some polynomial function $Q'\colon\RR^2\to\RR$. Indeed, if we write the homogeneous terms of degree $k$ of $R(x,y)$ in the form
$$
\sum_{j=0}^kc_j{\,}x^jy^{k-j} ~=~ (x-y){\,}\sum_{j=1}^k\bigg(\sum_{i=0}^{k-j}c_{k-i}\bigg){\,}x^{j-1}y^{k-j}+\bigg(\sum_{j=0}^kc_j\bigg){\,}y^k{\,},
$$
then we see that $R(x,x)=0$ if and only if $\sum_{j=0}^kc_j=0$. Thus, we have $F_2(x,y)=y+(x-y){\,}Q(x,y)$ for some polynomial function $Q\colon\RR^2\to\RR$ of degree $\geqslant 1$. Therefore, $F$ falls into case (ii).
%We note that if $F_n$ is of degree $1$ for some integer $n\geqslant 2$, that is, $F_n=M_n^z$ for some $z\in\RR$ (by Proposition~\ref{prop:lin-b}), then $F_{n+1}$ is of degree $1$ or constant by Proposition~\ref{prop:lin-a}. If, in addition, $\Delta_{n+1}^z=0$, then $F_{n+1}$ must be constant by Proposition~\ref{prop:lin-b}.
This completes the proof of the Main Theorem when $\RR$ is a field.

Let us now prove the Main Theorem when $\mathcal{R}$ is an infinite integral domain. Using the identification of polynomials and polynomial functions, we can extend every B-associative $\ast$-ary polynomial function over an infinite integral domain $\mathcal{R}$ to a $\ast$-ary polynomial function on the fraction field $\mathrm{Frac}({\RR})$ of $\RR$. The latter function is still B-associative since the B-associativity property for $\ast$-ary polynomial functions is defined by a set of polynomial equations on the coefficients of the polynomial functions. Therefore, every B-associative $\ast$-ary polynomial function $F$ over $\mathcal{R}$ is the restriction to $\mathcal{R}$ of a B-associative $\ast$-ary polynomial function $\overline{F}$ over $\mathrm{Frac}({\RR})$. The possible expressions for such a polynomial function $\overline{F}$ are given by the Main Theorem over $\mathrm{Frac}(\RR)$. Clearly, if $\overline{F}$ falls into case (ii), then so does $F$. If $\overline{F}$ falls into case (i), then there exist $z\in\mathrm{Frac}(\RR)$ and an integer $k\geqslant 1$ or $k=\infty$, with $k\leqslant\inf\{n\geqslant 1 : \Delta_n^z= 0\}$, such that $\overline{F}\in [M^z]_k$. If $k=1$, then $\overline{F}_n$ is constant for every $n\geqslant 1$. Therefore $F_n$ is also a constant (in $\RR$) for every $n\geqslant 1$ and hence $F$ falls into case (i). If $k\geqslant 2$, then $\overline{F}\in [M^z]_k$, where
$z=\overline{F}_2(1,0)=F_2(1,0)\in\RR$. For every integer $n<k$, we have
$$
\overline{F}_n(\bfx) ~=~ M_n^z(\bfx) ~=~ \sum_{i=1}^n(\Delta_n^z)^{-1}\, z^{n-i}(1-z)^{i-1}{\,}x_i{\,}.
$$
Since $\overline{F}_n$ is the extension of $F_n$, the coefficient $(\Delta_n^z)^{-1}\, z^{n-i}(1-z)^{i-1}$ of $x_i$ in $\overline{F}_n(\bfx)$ is in $\RR$ for $i=1,\ldots,n$. A straightforward induction shows that $(\Delta_n^z)^{-1}\, z^{n-j}\in\RR$ for $j=1,\ldots,n$. Therefore $\Delta_n^z$ is invertible in $\RR$ for every $n<k$ and hence $k\leqslant n(z)$. This shows that $F$ falls into case (i). The proof is now complete.

\section*{Acknowledgments}

This research is partly supported by the internal research project F1R-MTH-PUL-12RDO2 of the University of Luxembourg. J. Tomaschek is supported by the National Research Fund, Luxembourg (AFR 3979497), and cofunded under the Marie Curie Actions of the European Commission (FP7-COFUND).

%\bibliographystyle{abbrv}   % styles: plain, unsrt, alpha, abbrv, ieeetr, acm, siam, apalike, amsplain,...
%\bibliography{ReferencesMarichal}

\begin{thebibliography}{99}

\bibitem{Ant98}
C.~Antoine.
\newblock Les Moyennes.
\newblock Volume 3383 of {\em Que sais-je?}
\newblock Presse Universitaires de France, Paris, 1998.

%\bibitem{Bem26}
%G.~Bemporad.
%\newblock Sul principio della media aritmetica. (Italian).
%\newblock {\em Atti Accad. naz. Lincei} 3(6):87-91, 1926

\bibitem{deF31}
B.~de Finetti.
\newblock Sul concetto di media.
\newblock {\em Giornale dell' Instituto Italiano degli Attari} 2(3):369--396, 1931.

\bibitem{FodRou94}
J.~Fodor and M.~Roubens.
\newblock {\em Fuzzy preference modelling and multicriteria decision support}.
\newblock Kluwer, Dordrecht, 1994.

\bibitem{GraMarMesPap09}
M.~Grabisch, J.-L.~Marichal, R.~Mesiar, and E.~Pap.
\newblock {\em Aggregation functions}.
\newblock Encyclopedia of Mathematics and its Applications, vol. 127.
\newblock Cambridge University Press, Cambridge, 2009.

\bibitem{Kol30}
A.~N. Kolmogoroff.
\newblock Sur la notion de la moyenne. ({F}rench).
\newblock {\em Atti Accad. Naz. Lincei}, 12(6):388--391, 1930.

\bibitem{Nag30}
M.~Nagumo.
\newblock {\"U}ber eine klasse der mittelwerte. ({G}erman).
\newblock {\em Japanese Journ. of Math.}, 7:71--79, 1930.

\bibitem{Sch09}
R.~Schimmack.
\newblock Der Satz vom arithmetischen Mittel in axiomatischer Begr\"undung.
\newblock {\em Math. Ann.} 68:125--132, 1909.

\end{thebibliography}

\end{document}